\documentclass[11pt]{amsart}

\usepackage[T1]{fontenc}
\usepackage[dvipsnames]{xcolor}
\usepackage{mathtools,amssymb}
\usepackage{graphicx} 
\usepackage[mathscr]{euscript} 
\usepackage{enumitem}
\usepackage{fancyhdr}
\usepackage{comment}
\usepackage[
  a4paper,
  bindingoffset=0.2in,
  left=0.8in,
  right=1in,
  top=1in,
  bottom=1in,
  footskip=.25in,
  headheight=15pt
]{geometry}
\usepackage[
  pagebackref,
  colorlinks,
  allcolors=Mahogany
]{hyperref}
\usepackage[capitalize]{cleveref}

\newtheorem{theorem}{Theorem}[section]
\newtheorem{lemma}[theorem]{Lemma}
\newtheorem{proposition}[theorem]{Proposition}

\newtheorem*{corollary*}{Corollary}

\newtheorem{atheorem}{Theorem}

\newtheorem{acorollary}[atheorem]{Corollary}
\theoremstyle{definition}

\theoremstyle{remark}

\newtheorem{remark}[theorem]{Remark}

\newcommand{\subref}[2]{\hyperref[#2]{\ref{#1}.\ref{#2}}}

\numberwithin{equation}{section}
\newcommand{\R}{\mathbb{R}}
\newcommand{\Z}{\mathbb{Z}}

\newcommand{\prob}{\mathbb{P}}

\newcommand{\ang}[1]{\langle #1 \rangle}

\newcommand{\E}{\mathbb{E}}

\newcommand{\B}{\mathcal{B}}

\newcommand{\M}{\mathcal{M}}

\newcommand{\F}{\mathcal{F}}
\newcommand{\Hi}{\mathbb{H}}
\newcommand{\Ham}{\mathcal{H}}

\title{Predictable subordination, sharp martingale inequalities and applications}
\author{Francesco D'Emilio}
\thanks{\textbf{Funding Acknowledgment:} Francesco D'Emilio is partially supported by the Simons Foundation through a Simons Dissertation Fellowship.}
\begin{document}

\maketitle
\begin{abstract}
We introduce a new method for obtaining sharp $L^p$ estimates for martingales under predictable analogues of
differential subordination. By allowing suitable comparisons between continuous and jump variation adapted to the relevant range of \(p\), we retain sharp constants in settings where pathwise differential subordination fails. The key idea is to study together the continuous and jump
contributions arising from the appropriate Bellman functions.
Although these contributions need not be nonpositive separately, we
quantify their defects and show that they compensate at the predictable
level. This compensation mechanism is inspired by the author's earlier
work \cite{dunklmine}. As an application, we
substantially improve the explicit dimension-free bounds of
\cite{domelevo2026} for the Riesz vectors on the Hamming cube and on
$\Z^n$.
\end{abstract}

\section{Introduction}
Martingale inequalities have played a central role at the intersection
of probability and harmonic analysis. A foundational development was
Burkholder's work on martingale transforms \cite{Burkholder1966}, followed
by the square-function and maximal inequalities in
\cite{Burkholder1973,BurkholderDavisGundy1972,BurkholderGundy1970,Davis1970}.
These results became essential tools of stochastic analysis and revealed
a close connection between martingale transforms and singular integrals.
For historical accounts of this circle of ideas, we refer to
\cite{Ban2,osekowski}.
A further breakthrough was the determination of optimal constants
through what is now called Burkholder's method, which rests on the construction of suitable
Bellman functions. Following the developments in
\cite{Ban1,Burkholder1979,burk,Burkholder1987Stochastic,bur88}, Wang
\cite{W95} established the form
that is our starting point:
\[
    \|Y\|_p\leq(p^*-1)\|X\|_p,
    \qquad 1<p<\infty,
\]
where $p^*=\max\left\{p,p'\right\}$ and $p'=p/(p-1)$ is the dual exponent of $p$. Here $X,Y$ are c\`adl\`ag martingales taking values in a separable Hilbert
space $\Hi$, and
\[
    \|X\|_p:=\sup_{t\geq0}\|X_t\|_{L^p}.
\]
The estimate holds when $Y$ is differentially subordinated to $X$:
$|Y_0|\leq|X_0|$ almost surely and
\begin{equation}\label{class subord}
    [X]_t-[Y]_t
    \quad\text{is nonnegative and nondecreasing in }t.
\end{equation}
The constant $p^*-1$ is optimal. We also write $d[Y]\leq d[X]$ for
\eqref{class subord}. \\
Several extensions have explored which aspects of this hypothesis are
essential. Os\k{e}kowski \cite{OsekowskiECP2011} retained the sharp
constants under pathwise comparisons of accumulated or remaining
optional quadratic variation, with different hypotheses in the two
conjugate exponent ranges. Weighted and change-of-law inequalities were
developed in
\cite{BanuelosOsekowski2018Weighted,DomelevoPetermichl2019}, retaining
classical differential subordination. In another direction, weak
differential subordination in UMD spaces requires optional quadratic
variation comparisons and connects martingale inequalities with Banach-space
geometry and Fourier multipliers
\cite{Yaroslavtsev2018,Yaroslavtsev2020}. \\
The applications of martingale inequalities to harmonic analysis have
an equally long history. The martingale representation of singular
integrals, originating in the work of Gundy-Varopoulos \cite{GV},
combined with martingale inequalities, gives dimension-free
bounds for Riesz transforms and the Beurling transform
\cite{BanuelosJanakiraman2008,BW95}. The optimal constants for the
Euclidean Riesz vector and the Beurling transform remain unknown.
The same principle has subsequently been applied to L\'evy multipliers
and broader classes of Fourier multipliers
\cite{BanuelosBogdan2007,BanuelosBielaszewskiBogdan2011,
BanuelosOsekowski2012Fourier,GeissMontgomerySmithSaksman2010}, to Riesz
transforms and multipliers on manifolds and Lie groups
\cite{Li2008,BanuelosBaudoin2012,CarbonaroDragicevic2013Bakry,
ApplebaumBanuelos2014,BanuelosOsekowski2015Manifolds,
BanuelosBaudoinChenSire2021}, and to discrete singular integrals
\cite{LustPiquard1998,LustPiquard2004,BanuelosKwasnicki2019,
BanuelosKimKwasnicki2026}. \\

The present paper studies martingale comparisons at the level of predictable
compensators. While classical differential
subordination compares the continuous and jump parts separately, predictable subordination, which is generally a weaker hypothesis, can compare their combined
contributions. We show that suitable hypotheses of this kind preserve
sharp martingale estimates, and determine the exact exponent ranges in which
they imply any finite universal $L^p$ bound. This immediately finds applications in harmonic analysis, in particular to martingale representations for which
classical differential subordination fails, as we will discuss.
\subsection{Main results}
Let $(\Omega,\F,(\F_t)_{t\geq0},\prob)$ be a filtered probability space
satisfying the usual conditions. All martingales below are taken in
their c\`adl\`ag versions. Let $\Hi$ be a real separable Hilbert space,
with inner product $(\cdot,\cdot)$ and norm $|\cdot|$. For a $\Hi$-valued martingale $X$, let
\[
    [X]_t=\sum_{n\geq1}[(X-X_0,e_n)]_t
\]
be its scalar quadratic variation, where $\{e_n\}_n$ is an orthonormal basis of $\Hi$. Decomposing
\[
    X=X_0+X^c+X^d,
\]
where $X^c,X^d$ are the continuous and purely discontinuous parts of $X$
starting at zero, we have
\[
    [X]^c=[X^c],
    \qquad
    [X^d]_t=\sum_{0<s\leq t}|\Delta X_s|^2.
\]
As observed in \cite{W95}, differential subordination is equivalent
to the following pathwise conditions, holding almost surely:
\begin{enumerate}[label=\textup{(\roman*)}]
    \item $|Y_0|\leq|X_0|$;
    \item $|\Delta Y_t|\leq|\Delta X_t|$ for all $t>0$;
    \item $[X^c]-[Y^c]$ is nonnegative and nondecreasing.
\end{enumerate}
In particular, jump variation cannot compensate for a failure of the
continuous comparison. Whenever a predictable bracket $\langle M\rangle$ is used, we assume
that $M-M_0$ is locally square-integrable and take
$\langle M\rangle$ to be the predictable compensator of $[M]$, starting
at zero. For continuous martingales, optional and predictable quadratic
variations coincide. Consequently, whenever the brackets are defined,
\[
    \langle X\rangle
    =\langle X^c\rangle+\langle X^d\rangle
    =[X^c]+\langle X^d\rangle.
\]
We say that $Y$ is \emph{predictably subordinated} to $X$ if
$|Y_0|\leq|X_0|$ almost surely and
\begin{equation}\label{pred subord}
    \langle X\rangle_t-\langle Y\rangle_t
    \quad\text{is nonnegative and nondecreasing in }t.
\end{equation}
We write $d\langle Y\rangle\leq d\langle X\rangle$ for\eqref{pred subord}. Differential subordination implies predictable subordination,
since taking predictable compensators preserves the order of the
associated increasing measures. The two notions agree for continuous
martingales, but differ in the presence of jumps. \\ Our first result shows that, in the appropriate range of $p$, passing from the optional to the predictable bracket does not produce any loss in the sharp constant. 
\begin{atheorem} \label{main theorem1}
Let $X$ and $Y$ be $\Hi$-valued martingales such that $|Y_0|\leq |X_0|$. For $p \geq 2$, if $X$ is càdlàg, $Y$ is continuous and $d \ang{Y} \leq d \ang{X}$, then 
\begin{equation} \label{A.1} \|Y\|_p \leq (p-1) \|X\|_p.\end{equation}
For $1<p\leq2$, if $Y$ is merely càdlàg and $d\ang{Y} \leq d\ang{X^c}$, then 
\begin{equation}\label{A.2}\|Y\|_p \leq \frac{1}{(p-1)}\|X\|_p.\end{equation}
In both cases, the constants are sharp.
\end{atheorem}
The two estimates are dual at conjugate exponents and the constants are already known to be sharp when both $X$ and $Y$ are continuous, a case covered by our hypotheses. A stronger comparison, in which the dominating bracket comes entirely from the opposite part of the continuous/purely discontinuous decomposition, yields a smaller optimal constant. 
\begin{atheorem}\label{main theorem2}
Let $X$ and $Y$ be $\Hi$-valued martingales such that $|Y_0|\leq |X_0|$. For $p\geq 2$, if $X$ is càdlàg, $Y$ is continuous and $d\langle Y\rangle\leq d\langle X^d\rangle$, then
\begin{equation}\label{B.1}
\|Y\|_p\leq z_{p'}^{-1}\|X\|_p.
\end{equation}
If $1<p\leq2$, let $Y$ be purely discontinuous and suppose that $d\langle Y\rangle\leq d\langle X^c\rangle.$ Then
\begin{equation}\label{B.2}
\|Y\|_p\leq z_p^{-1}\|X\|_p.
\end{equation}
Here, for $1<q\leq2$, $z_q$ denotes the largest positive zero of the parabolic cylinder function $D_q$. In both cases the constants are sharp and the inequalities are strict if $p\ne2$ and $0<\|X\|_p<\infty$.
\end{atheorem}
As before, the two statements are dual to each other. The result is proved using the parabolic-cylinder Bellman function introduced by Os\k{e}kowski in \cite{OsekowskiEJP2011} in his study of sharp inequalities for orthogonal martingales. The connection with the orthogonal setting in the range $1<p\leq2$ is explained by the fact that $X^c$, whose predictable variation controls that of $Y$, is automatically orthogonal to $Y$ since the latter is purely discontinuous. Accordingly, the covariance terms disappear from the continuous part of the Bellman drift. The lower-range estimate may also be viewed as a jump analogue of the classical stopping-time inequality of Davis \cite{Davis1976}. Indeed, the Bellman function used by Os\k{e}kowski in this range is the one naturally associated with the Davis stopping-time problem. This connection also explains the sharpness of the constant $z_p^{-1}$. \\
We finally note that the assumptions of Theorems \ref{main theorem1} and \ref{main theorem2} are range-optimal and sharp with respect to the continuous/purely discontinuous decomposition: if either of the two bracket conditions is retained in the opposite range of \(p\), no \(L^p\) inequality with a finite universal constant can hold. We will discuss this in \Cref{sec: discussion of hypothesis}. \\

Our main application concerns the Riesz vectors on the Hamming
cube $\Ham_n$ and on $\Z^n$. Dimension-free Riesz estimates in discrete and noncommutative settings
have been studied in
\cite{LustPiquard1998,LustPiquard2004,Naor2016,JMP2018}. More recently, Domelevo, Ivanisvili, Petermichl, and Volberg \cite{domelevo2026}
gave a Bellman-function proof with the explicit bound
\[
    \|\mathbf Rf\|_{L^p(\Omega;\ell^2)}
    \leq1024(p-1)\|f\|_{L^p(\Omega)},
    \qquad 2\leq p<\infty,
\]
for $\Omega=\Ham_n,\Z^n$, and proved that linear growth in $p$ is
optimal on the Hamming cube. We apply Theorem \ref{main theorem2} to a martingale representation of the Riesz vector on the cube. For a given function $f$, the associated Poisson martingale has a Brownian part coming from the vertical coordinate and jumps generated by horizontal coordinate flips. The martingale representing the Riesz vector is continuous, and its predictable bracket satisfies the hypothesis of Theorem \ref{main theorem2}, although classical differential subordination fails in general. On the lattice $\Z^n$, we use the $2n$-dimensional vector of forward and backward Riesz transforms, as in \cite{domelevo2026}.
\begin{atheorem}
Let $p\geq2$ and $\Ham_n$ be the Hamming cube. We have
\[
\|\mathbf R f\|_{L^p(\Ham_n;\ell_n^2)}
\leq
\frac{2}{z_{p'}}\|f\|_{L^p(\Ham_n)}.
\]
For $\Z^n$, we have 
\[
\|\mathbf R f\|_{L^p(\Z^n;\ell_n^2)}
\leq
\frac{2\sqrt2}{z_{p'}}\|f\|_{L^p(\Z^n)}.
\]
\end{atheorem}
The additional factors come from the representation of the Riesz vector used. One can obtain slightly better results by interpolating with $p=2$. Notice that $\|\textbf{R}\|_{L^2(\Z^n) \to L^2(Z^n,\ell^2_{2n})}=\sqrt{2}.$
\begin{acorollary}\label{cor}
For any $q \geq p>2$ let $\theta(p,q)=(q(p-2))/(p(q-2)).$ Then 
\begin{align*}
\|\textbf{R}\|_{L^p(\Ham_n; \ell^2_n)} &\leq \inf_{q \geq p} \left( \frac{2} {z_{q'}} \right)^{\theta(p,q)},\\ \|\mathbf R\|_{L^p(\Z^n;\ell_{2n}^2)}
& \leq
\sqrt2\left[
 \inf_{p\leq q<\infty}
 \left(\frac{2}{z_{q'}}\right)^{\theta(p,q)}
\right].\end{align*}
\end{acorollary}

 Corollary \ref{cor} gives a mild improvement for $p$ close to $2$. Asymptotically, we have
\[
 \frac2{z_{p'}}\sim\sqrt{\frac8\pi}\,(p-1)
 \approx1.596(p-1),\qquad p\to\infty.
\]
Dimension-free bounds for Riesz vectors fail when $1<p<2$;
see \cite{LustPiquard1998}. The same predictable bracket
comparison for the Poisson martingales remains valid independently of
$p$, but it cannot imply
a universal estimate in that range. The lower-range hypotheses require
a different comparison of continuous and jump variation and do not
hold for these representations in general. This explains why the
exponent restrictions in the martingale results agree with the
obstruction in the applications. We expect this approach to extend verbatim to
the locally compact abelian groups considered in
\cite{LustPiquard2004}.
\subsection{Structure of the paper}
Section~\ref{sec: discussion of hypothesis} discusses the hypotheses and proves the duality between the two exponent ranges. Section~\ref{sec:reductions} records the reductions and the common Bellman strategy. Theorems~\ref{main theorem1} and \ref{main theorem2} are proved in Sections~\ref{sec:proof-theorem-A} and \ref{sec:proof-theorem-B}, respectively. The applications to discrete Riesz vectors are presented in Section~\ref{sec:applications}.
\subsection{Acknowledgments} I would like to thank José Conde Alonso and Guillermo Rey for organizing the summer school “New Trends in Harmonic Analysis,” which I attended in Madrid in May 2026. I believe the school has already had, and will continue to have, a tremendous impact on our community (greater, I like to think, than that of AI) and in particular on this project. I am also grateful to Stefanie Petermichl, whose course at the summer school introduced me to the Riesz vector problem on the Hamming cube and to her work with her collaborators, for being very generous with her time and very kind in answering my questions.
\section{Discussion on the subordination hypotheses}\label{sec: discussion of hypothesis}
We start by establishing three points:
\begin{enumerate}
    \item Predictable subordination is strictly weaker than
    classical differential subordination and allows comparisons
    between continuous and jump variation.

    \item The hypotheses in the $p\geq2$ parts of
    Theorems \ref{main theorem1} and \ref{main theorem2}
    cannot yield a universal $L^p$ inequality when $1<p<2$.
    Likewise, the hypotheses in their $1<p\leq2$ parts
    cannot yield such an inequality when $p>2$.

    \item For both theorems, the upper and lower range estimates
    are dual at conjugate exponents, and this duality preserves
    the optimal constants.
\end{enumerate}
The first two points are cleanly explained by the following classical example. 
For $t \in [0,1]$, let $B_t$ be a standard Brownian motion and let $N_t$ be a Poisson process with rate $\lambda >0$ such that $N_0=0$. Consider the normalized compensated Poisson process $$\widetilde{N}^\lambda_t= \frac{N_t-\lambda t}{\sqrt \lambda}. $$ Then we have
\[
 \ang{B}_t=[B]_t=\ang{\widetilde N^\lambda}_t=t,
 \qquad
 [\widetilde N^\lambda]_t=\frac{N_t}{\lambda}.
\]
Thus $B$ is predictably subordinated to $\widetilde N^\lambda$,
but differential subordination fails, since $B$ has
nonzero continuous quadratic variation while
$\widetilde N^\lambda$ is purely discontinuous. Also, for $1<p<\infty$
\[
 \|\widetilde N^\lambda\|_p
 \sim\lambda^{1/p-1/2},
 \qquad \lambda\to0,
\]
whereas $\|B\|_p$ is positive and independent of $\lambda$.
For $1<p<2$, take $X=\widetilde N^\lambda$ and $Y=B$.
The hypotheses of the upper-range parts of both theorems hold,
but $\|X\|_p\to0$ while $\|Y\|_p$ remains fixed.
For $p>2$, interchange the roles and take
$X=B$ and $Y=\widetilde N^\lambda$.
The hypotheses of the lower-range parts hold,
but now $\|Y\|_p\to\infty$ while $\|X\|_p$ remains fixed.
In both cases,
\[
 \frac{\|Y\|_p}{\|X\|_p}\longrightarrow\infty.
\]
Hence, neither set of hypotheses yields a finite
universal constant in the opposite exponent range. Recall that $(\cdot,\cdot)$ denotes the inner product of $\Hi$ and $|\cdot|$ its norm. We now prove the duality. 
\begin{proposition}\label{duality}
Under the respective hypotheses, we have that \eqref{A.1} holds for $p \geq 2$ if and only if \eqref{A.2} holds for $1<p\leq 2$. Analogously, \eqref{B.1} holds for $p \geq 2$ if and only if \eqref{B.2} holds for $1<p \leq 2$.
\end{proposition}
\begin{proof} 
Assume first that every martingale starts at zero; the adaptation for the general case is given at the end. Fix $1<p\leq2$ and let $q=p'$. We first show that if \eqref{A.1} holds for
$q\geq2$, then \eqref{A.2} holds for $p$. Let $Z\in L^q(\Omega;\Hi)$ with
$\|Z\|_q=1$ and set
$$
Z_t:=\E[Z\mid\F_t].
$$
Localize, if needed, so that every bracket process is square integrable. By integration by parts and the fact that $Y_0=0$,
$$
\E(Y_T,Z)
=
\E(Y_T,Z_T)
=
\E[Y,Z]_T
=
\E\ang{Y,Z}_T.
$$

Set
$
A_t:=\ang{X^c}_t+\ang{Z}_t.
$
By Radon-Nikodym, there exist
scalar predictable processes $x,y,z,c$ such that
$$
\ang{X^c}_t=\int_0^t x_s\,dA_s,\qquad
\ang{Y}_t=\int_0^t y_s\,dA_s,\qquad
\ang{Z}_t=\int_0^t z_s\,dA_s,\qquad 
\ang{Y,Z}_t=\int_0^t c_s\,dA_s.
$$
Predictable subordination implies
$0\leq y_s\leq x_s$ almost surely, and Kunita-Watanabe gives
$$
c_s^2\leq y_sz_s\leq x_sz_s.
$$
Define
$$
a_s:=
\begin{cases}
c_s/x_s, & x_s>0,\\
0, & x_s=0.
\end{cases}
$$
Since $c_s=0$ on $\{x_s=0\}$, we have
$$
\ang{Y,Z}_t
=
\int_0^t a_s\,d\ang{X^c}_s, \qquad
\int_0^t a_s^2\,d\ang{X^c}_s
\leq
\ang{Z}_t.
$$

Therefore, the $\Hi$-valued continuous martingale defined by
$$
W_t:=\int_0^t a_s\,dX_s^c
$$
satisfies
$
\ang{W}_t
\leq
\ang{Z}_t
$
and \eqref{A.1} implies
$
\|W\|_q
\leq
(q-1).
$ By integration by parts,
$$
|\E(Y_T,Z)|=|\E(X_T,W_T)|
\leq
\|X_T\|_p\|W_T\|_q
\leq
(q-1)\|X_T\|_p,
$$
Taking the supremum over $Z$ and then over $T$ proves \eqref{A.2}. Following the same construction, one can prove that \eqref{A.2} implies \eqref{A.1} and we omit the details. \\
Now, let again $1<p\leq2$ and $q=p'$. Assuming \eqref{B.1}, take $Z\in L^q(\Omega;\Hi)$ as before. Since $Y$ is purely discontinuous, $\ang{Y,Z}=\ang{Y,Z^d}$ and, arguing as above, the Kunita-Watanabe inequality and $\ang{Y}\leq\ang{X^c}$ yield a predictable scalar process $a$ such that
$$
\ang{Y,Z}_t=\int_0^t a_sd\ang{X^c}_s,
\qquad
\int_0^t a_s^2 d\ang{X^c}_s\leq\ang{Z^d}_t.
$$
Thus $W_t:=\int_0^t a_sdX_s^c$ is continuous and satisfies $\ang{W}\leq\ang{Z^d}$. By \eqref{B.1},
$$
\|W\|_q\leq z_{q'}^{-1}\|Z\|_q=z_p^{-1}\|Z\|_q.
$$
As in the previous duality argument, integration by parts gives $\E(Y_T,Z_T)=\E(X_T,W_T)$, and \eqref{B.2} follows by duality. The converse implication follows in the same way, interchanging the continuous and purely discontinuous parts.

The same construction handles nonzero initial values. In each direction,
replace $W_0=0$ by
\[
 W_0=\mathbf1_{\{|X_0|>0\}}
       \frac{(Y_0,Z_0)}{|X_0|^2}X_0.
\]
Then $|W_0|\leq|Z_0|$ and $(X_0,W_0)=(Y_0,Z_0)$, so both the
initial domination and the integration-by-parts identity are preserved.
The rest follows as before.
\end{proof}
\section{Reductions and the general strategy}\label{sec:reductions}
In the rest of the paper, we will assume that $\Hi$ is a finite dimensional Hilbert space; in particular we will write $\Hi=\R^n $ for some $n$. For the approximation argument to general infinite dimensional Hilbert spaces follows the usual
procedure in \cite{W95}.
To emphasize the main ideas of the
proofs, we also omit the usual regularization procedure for the Bellman
functions. The functions involved are smooth away from their singular sets $S_\B$, and
the relevant differential and finite-difference inequalities are preserved under
convolution. Thus, one may first apply the argument to a smooth mollification and
then pass to the limit in the standard way. This procedure has been carried out
in full detail in many related works, and in particular in a recent paper of the
author \cite{dunklmine}, where a similar approach was used; we refer there for
further details. Finally, in view of the duality arguments above, for each pair
of estimates it suffices to establish the corresponding Bellman estimate in one
of the two exponent ranges; the estimate in the dual range then follows
immediately. \\
\subsection{Jump measure and compensator} We recall some useful classical facts, for which we refer to \cite{protter, cohell}. 
Let $\mu^{X,Y}$ be the joint jump measure of $(X,Y)$
$$
\mu^{X,Y}(ds,dh,dk)
=
\sum_{s>0}
\mathbf 1_{\{(\Delta X_s,\Delta Y_s)\neq(0,0)\}}
\delta_{(s,\Delta X_s,\Delta Y_s)}(ds,dh,dk),
$$
and let $\nu^{X,Y}$ denote its predictable compensator. In other words, for every predictable integrand
\[
 \E\int H(s,h,k)\,\mu^{X,Y}(ds,dh,dk)
 =\E\int H(s,h,k)\,\nu^{X,Y}(ds,dh,dk)
\]
holds when both sides are well defined. In particular, for locally square-integrable martingales
$$
[X^d]_t
=
\int_{(0,t]\times\Hi\times\Hi}
|h|^2\mu^{X,Y}(ds,dh,dk),
\qquad
[Y^d]_t
=
\int_{(0,t]\times\Hi\times\Hi}
|k|^2\mu^{X,Y}(ds,dh,dk),
$$
and
$$
[X^d,Y^d]_t
=
\int_{(0,t]\times\Hi\times\Hi}
(h,k)\mu^{X,Y}(ds,dh,dk).
$$
Passing to predictable compensators yields instead
$$
\ang{X^d}_t
=
\int_{(0,t]\times\Hi\times\Hi}
|h|^2\,\nu^{X,Y}(ds,dh,dk),
\qquad
\ang{Y^d}_t
=
\int_{(0,t]\times\Hi\times\Hi}
|k|^2\,\nu^{X,Y}(ds,dh,dk),
$$
as well as
$$
\ang{X^d,Y^d}_t
=
\int_{(0,t]\times\Hi\times\Hi}
(h,k)\,\nu^{X,Y}(ds,dh,dk).
$$
\subsection{Outline of the proofs}
Let $\B=\B_p$ be the Bellman function associated with one of the main results,
where $p$ is in the appropriate range specified below, and let $\mathcal O$
denote the corresponding Bellman obstacle. After localization and
regularization, It\^o's formula gives
$$
\E\B(X_T,Y_T)
=
\E\B(X_0,Y_0)+\E\mathcal N_T+\E\mathcal I_1+\E\mathcal I_2,
$$
where $\B(X_0,Y_0)\leq0$, $\mathcal N$ is a local martingale starting from
zero, and $\mathcal I_1$ and $\mathcal I_2$ denote respectively the continuous
and jump contributions. Set $\M=(X^c,Y^c)$ and let $d\mathbf C_\M(s)$ be its matrix-valued predictable covariance measure. In block form,
$$
d\mathbf C_\M(s)
=
\begin{pmatrix}
 d\mathbf C_{X^c}(s) & d\mathbf C_{X^c,Y^c}(s)\\
 d\mathbf C_{Y^c,X^c}(s) & d\mathbf C_{Y^c}(s)
\end{pmatrix},
$$
where
\begin{align*}
d\mathbf C_{X^c}(s)&=\bigl(d\langle X^{c,i},X^{c,j}\rangle_s\bigr)_{i,j},
\qquad
 d\mathbf C_{Y^c}(s)=\bigl(d\langle Y^{c,i},Y^{c,j}\rangle_s\bigr)_{i,j}, \\ 
 d\mathbf C_{X^c,Y^c}(s)&=\bigl(d\langle X^{c,i},Y^{c,j}\rangle_s\bigr)_{i,j},
\qquad
 d\mathbf C_{Y^c,X^c}(s)=d\mathbf C_{X^c,Y^c}(s)^{\mathsf T}.
\end{align*}
The continuous and jump contributions can be written as
\begin{align*}
\mathcal I_1
&=
\frac12\int_0^T
\operatorname{Tr}\left(
D^2\B(X_{s-},Y_{s-})d\mathbf C_\M(s)
\right), \\
\mathcal I_2 &=\sum_{0<s\leq T}\mathcal J_\B(X_{s-},Y_{s-};\Delta X_s,\Delta Y_s)
\end{align*}
where $$D^2\B(x,y)
=
\begin{pmatrix}
D^2_{xx}\B(x,y) & D^2_{xy}\B(x,y)\\
D^2_{yx}\B(x,y) & D^2_{yy}\B(x,y)
\end{pmatrix}, $$ and 
$$
\mathcal J_{\B}(x,y;h,k)
:=
\B(x+h,y+k)-\B(x,y)
-D_x\B(x,y)h-D_y\B(x,y)k.
$$
In the classical subordination case, the continuous
and jump contributions are controlled separately and are both nonpositive. Here
this need not be the case and we need to study the two terms together using the identities above. More precisely, the compensation formula gives
$$
\E\mathcal I_2
=
\E\int_{(0,T]\times\Hi\times\Hi}
\mathcal J_{\B}(X_{s-},Y_{s-};h,k)\,
\nu^{X,Y}(ds,dh,dk).
$$

In each of the cases considered below the continuous estimate
takes the form
$$
\E\mathcal I_1
\leq
\E\int_{(0,T]\times\Hi\times\Hi}
\mathcal C_{\B}(X_{s-},Y_{s-};h,k)\,
\nu^{X,Y}(ds,dh,dk)
$$
for a suitable function $\mathcal C_{\B}$.
The proof is therefore reduced to showing that
$$
\mathcal C_{\B}(x,y;h,k)+\mathcal J_{\B}(x,y;h,k)\leq0
$$
which implies that 
$$
\E(\mathcal I_1+\mathcal I_2)\leq0.
$$
Together with $\B(X_0,Y_0)\leq0$ and the obstacle inequality
$\mathcal O\leq\B$, this yields the desired estimate.

\section{Proof of Theorem \ref{main theorem1}}
\label{sec:proof-theorem-A}

By \cref{duality}, it is enough to prove the estimate in the range $p\geq 2$. In view of the reductions discussed above, we work in $\mathbb R^n$ and carry out the differential computations away from the singular sets of the Bellman function. Let $X$ be an $\mathbb R^n$-valued càdlàg martingale and let $Y$ be an $\mathbb R^n$-valued continuous martingale such that
$
d\langle Y\rangle\leq d\langle X\rangle.
$
Set
$$
C_p:=p\left(1-\frac1p\right)^{p-1},
\qquad
K_p:=C_pp(p-1),
$$
and consider the classical Burkholder function
$$
\B(x,y):=
C_p\bigl(|y|-(p-1)|x|\bigr)(|x|+|y|)^{p-1},
\qquad x,y\in\mathbb R^n.
$$
We recall the standard obstacle inequality
$$
\B(x,y)\geq |y|^p-(p-1)^p|x|^p.
$$

Let $\mathcal I_1$ and $\mathcal I_2$ denote respectively the continuous and jump contributions in It\^o's formula. Since $Y$ is continuous, the jumps of $(X,Y)$ are of the form $(h,0)$. Then
\begin{align*}
\mathcal I_1
&=
\frac12\int_0^T
\operatorname{Tr}\left(
D^2\B(X_{s-},Y_s)\,d\mathbf C_\M(s)
\right), \\
\mathcal I_2
&=
\sum_{0<s\leq T}
\mathcal J_\B(X_{s-},Y_s;\Delta X_s),
\end{align*}
where
\begin{align*}
\mathcal J_\B(x,y;h) =
\B(x+h,y)-\B(x,y)-\bigl(\nabla_x\B(x,y),h\bigr).    
\end{align*}
If $\nu^X$ denotes the predictable compensator of the jump measure of $X$, then 
$$
\mathbb E\mathcal I_2
=
\mathbb E\int_{(0,T]\times\mathbb R^n}
\mathcal J_\B(X_{s-},Y_s;h)\,\nu^X(ds,dh).
$$

\subsection{The Bellman estimates}

Define
$$
\Lambda_p(x,y):=
K_p(|x|+|y|)^{p-2}.
$$
For $p>2$ we set $\Lambda_p(0,0)=0$, while $\Lambda_2\equiv2$. We first record the following fact:
\[
D^2\B(x,y)
+
\Lambda_p(x,y)
\begin{pmatrix}
I_n&0\\
0&-I_n
\end{pmatrix}
\leq0.
\]
This estimate is implicit in Burkholder's argument and can indeed be deduced from property $(c)$ in \cite{W95}, page 527. Since $d\mathbf C_M$ is positive semidefinite,
\begin{align*}
2\mathcal I_1
&\leq
\int_0^T
\Lambda_p(X_{s-},Y_s)
\operatorname{Tr}\left[
\begin{pmatrix}
-I_n&0\\
0&I_n
\end{pmatrix}
d\mathbf C_\M(s)
\right] =
\int_0^T
\Lambda_p(X_{s-},Y_s)
\bigl(
d\langle Y\rangle_s-d\langle X^c\rangle_s
\bigr).
\end{align*}
Finally, $d\ang{X}=d\ang{X^c}+d\ang{X^d}$ and predictable subordination give 
\begin{align}
2\mathcal I_1 \leq
\int_0^T
\Lambda_p(X_{s-},Y_s)\,d\langle X^d\rangle_s.
\label{eq:I1-defect}
\end{align}

We next turn to the jump contribution. Here the relevant estimate is a finite-difference inequality; its proof reduces to a one-variable calculation by the radiality of $\B$ in its first argument.
\begin{lemma}
\label{lem:jump-burkholder}
For every $x,y,h\in\mathbb R^n$,
\[
\mathcal J_\B(x,y;h)
\leq
-\frac12\Lambda_p(x,y)|h|^2.
\]\end{lemma}
\begin{proof}
For $p=2$ the claim trivially holds with equality. Hence, assume that $p>2$ and fix $\Lambda= \Lambda_p(x,y)$.  Since
$
\nabla_x\B(x,y)=-\Lambda x,
$
the desired inequality corresponds to \[\B(x+h,y)-\B(x,y)- \frac \Lambda 2 (|x|^2-|x+h|^2) \leq 0.\]
Fix $z:=|y|$ and denote $r:=|x|$. We have that $\B(x,y)=\Phi(r)$, where
$$
\Phi(t):=C_p\bigl(z-(p-1)t\bigr)(t+z)^{p-1}, \quad t\geq 0.
$$
 For $t\geq0$, define
$$
F(t):=\Phi(t)-\Phi(r)+\frac12\Lambda(t^2-r^2).
$$
For $t=|x+h|$, we have $F(t)=\mathcal J_{\B}(x,y;h)+\frac12\Lambda|h|^2$, so we only need to prove that $F(t) \leq 0$. Recalling
$$
\Phi'(t)=-K_pt(t+z)^{p-2}, \qquad \Lambda:=\Lambda_p(x,y)=K_p(r+z)^{p-2}$$
 we have that 
$$
F'(t)
=
\Phi'(t)+\Lambda t
=
K_pt\left((r+z)^{p-2}-(t+z)^{p-2}\right).
$$
Therefore $F$ attains its maximum at $t=r$ for $t \geq 0$, since
$$
F'(t)\geq0\quad\text{for }0\leq t\leq r,
\qquad
F'(t)\leq0\quad\text{for }t\geq r,
$$
and the claim follows since $F(r)=0.$
\end{proof}
By Lemma \ref{lem:jump-burkholder} and the compensation formula,
\begin{align}
2\mathbb E\mathcal I_2
&\leq
-\mathbb E\int_{(0,T]\times\mathbb R^n}
\Lambda_p(X_{s-},Y_s)|h|^2\,\nu^X(ds,dh) =
-\mathbb E\int_0^T
\Lambda_p(X_{s-},Y_s)\,d\langle X^d\rangle_s.
\label{eq:I2-defect}
\end{align}

Combining \eqref{eq:I1-defect} and \eqref{eq:I2-defect} yields
$$
\mathbb E(\mathcal I_1+\mathcal I_2)\leq0.
$$
Returning to It\^o's formula and using the obstacle inequality, we conclude that
$$
\E|Y_T|^p-(p-1)^p\E|X_T|^p
\leq
\E \B(X_T,Y_T)
\leq0.
$$
This proves the result for $p \geq 2$ and \cref{duality} gives the range $1<p \leq 2$.


\section{Proof of Theorem \ref{main theorem2}}\label{sec:proof-theorem-B}
We first prove the lower-range estimate, fixing $1<p\leq2$; the upper range follows by Proposition~\ref{duality}. Strictness and sharpness are proved at the end of the section.
\subsection{The Bellman function} Let $D_p$ be the parabolic cylinder function. Define $$\phi_p(s):= e^{s^2/4}D_p(s)$$ and let $z_p$ be the largest positive root of $D_p$. We recall the following facts. 
\begin{lemma}[{\cite[Lemma 3.1]{OsekowskiEJP2011}}]\label{parabcyl}
For $1<p\leq2$, the function $\phi_p$ satisfies
\[
 \phi_p''(s)-s\phi_p'(s)+p\phi_p(s)=0,\qquad
 \phi_p'(s)=p\phi_{p-1}(s).
\]
Consequently,
\[
 \begin{gathered}
 \phi_p''(s)=p(p-1)\phi_{p-2}(s),\qquad
 \phi_p'''(s)=p(p-1)(p-2)\phi_{p-3}(s),\\
 \phi_p(s)=s\phi_{p-1}(s)-(p-1)\phi_{p-2}(s).
 \end{gathered}
\]
For $s\geq z_p$, we have $$\phi_p(s)\geq0, \quad \phi_p'(s)>0, \quad 
\phi_p''(s)>0, \quad \phi_p'''(s)\leq0.$$ The last inequality
is strict when $p<2$.
\end{lemma}
The scalar Bellman function introduced in \cite{OsekowskiEJP2011} takes the form 
\begin{equation}\label{ose bellman}
   V_p(x,y)= \begin{cases}
        \alpha_py^p \phi_p(x/y), \quad 0<y<x/z_p, \\ y^p-(x/z_p)^p, \quad y \geq x/z_p, 
    \end{cases} \qquad \alpha_p:=-\frac{1}{z_p \phi_{p-1}(z_p)}<0.
\end{equation}
We set $V_p(x,0)=\alpha_px^p$ for $x>0$ and $V_p(0,0)=0$ by continuity. \\

In the rest of the proof we will call the \emph{parabolic-cylinder region} the set $\{(x,y), \ 0<y<x/z_p\}$ and the \emph{obstacle region} the set $\{(x,y), \ y > x/z_p\}.$ Our Bellman function is then 
\[ \B(x,y)= V_p(|x|,|y|), \qquad (x,y) \in \R^{2n}.\] The function $\B$ is $C^1$ everywhere and, as observed in \cite{OsekowskiEJP2011}, is of class $C^2$ away from $$ S:= \{|x| =z_p |y| \} \cup \{|x|=0\} \cup \{ |y|=0 \} \subset \R^{2n}.$$  

Before studying the drift terms, we record the following formulas.
\begin{lemma}\label{radial derivatives}
For $(x,y) \in \R^{2n} \setminus S$, set $r=|x|$, $s=|y|$. In the parabolic-cylinder region,
writing $q=r/s$, we have
$$
\partial_rV_p(r,s)=\alpha_p s^{p-1}\phi_p'(q),
\qquad
\partial_{rr}V_p(r,s)=\alpha_p s^{p-2}\phi_p''(q).
$$
Moreover, the relations for $\phi_p$ in Lemma~\ref{parabcyl} give
$$
\begin{aligned}
\partial_sV_p(r,s)
&=-\alpha_p s^{p-1}\phi_p''(q)
=-s\partial_{rr}V_p(r,s).
\end{aligned}
$$  
In the obstacle region,
$$
\partial_rV_p(r,s)=-pz_p^{-p}r^{p-1},
\qquad
\partial_{rr}V_p(r,s)=-p(p-1)z_p^{-p}r^{p-2},
\qquad
\partial_sV_p(r,s)=ps^{p-1}.
$$
In particular, $\partial_rV_p\leq0$ and $\partial_sV_p>0$
in both regions.
\end{lemma}

\subsection{The drift terms}
Let $Y$ be purely discontinuous and $X$ c\`adl\`ag, with $d\langle Y\rangle\leq d\langle X^c\rangle$. Then
\[
 d\mathbf C_\M(s)=
 \begin{pmatrix}d\mathbf C_{X^c}(s)&0\\0&0\end{pmatrix}.
\]
We now deal with the continuous drift $\mathcal I_1$.
\begin{proposition} \label{continuous drift pt.2}
For $(x,y) \in \R^{2n} \setminus S$, set $r=|x|$, $s=|y|$. Then for every $h \in \R^n$ we have    
\[
D^2_{xx}\B(x,y)[h,h] \leq-\Lambda_p(x,y)|h|^2,
\]
where $$
\Lambda_p(x,y):=\frac{\partial_sV_p(r,s)}s>0.
$$
\end{proposition}
As a consequence of this, since $d\mathbf C_{X^c}$ is positive semidefinite and
$\operatorname{Tr}(d\mathbf C_{X^c})=d\langle X^c\rangle$ we have
\begin{align*}
\mathcal I_1
=\frac12\int_0^T\operatorname{Tr}
\bigl(D^2_{xx}\B(X_{s-},Y_{s-})\,d\mathbf C_{X^c}(s)\bigr)
\leq-\frac12\int_0^T
\Lambda_p(X_{s-},Y_{s-})\,d\langle X^c\rangle_s.
\end{align*}
We now prove Proposition~\ref{continuous drift pt.2}.
\begin{proof}
For $p=2$, the claim is immediate, so assume
$1<p<2$. Since $(x,y)\notin S$, both $r$ and $s$ are positive.
Radial differentiation gives
$$
D^2_{xx}\B(x,y)[h,h]
=
\partial_{rr}V_p(r,s)\frac{(x,h)^2}{r^2}
+
\frac{\partial_rV_p(r,s)}r
\left(|h|^2-\frac{(x,h)^2}{r^2}\right).
$$
The result follows provided we show
$$
\frac{\partial_rV_p(r,s)}r
\leq \partial_{rr}V_p(r,s)
\leq -\Lambda_p(x,y)<0.
$$

In the parabolic-cylinder region, put $q=r/s$. Using the identities in Lemma~\ref{parabcyl} we get
$$
\frac{\partial_rV_p(r,s)}r
=\alpha_p s^{p-2}\frac{\phi_p'(q)}q,
\qquad
\partial_{rr}V_p(r,s)
=\alpha_p s^{p-2}\phi_p''(q)
=-\Lambda_p(x,y).
$$
Thus only the first comparison remains to be checked.
Differentiating the equation for $\phi_p$ gives
$$
q\phi_p''(q)
=(p-1)\phi_p'(q)+\phi_p'''(q)
\leq \phi_p'(q),
$$
where we used $p-1<1$, $\phi_p'>0$ and $\phi_p'''\leq0$
from Lemma~\ref{parabcyl}. Multiplying by
$\alpha_p s^{p-2}/q<0$ proves the required comparison.
Moreover, $\phi_p''>0$ implies
$\partial_{rr}V_p<0$. \\

In the obstacle region, direct differentiation gives
$$
\frac{\partial_rV_p(r,s)}r=-pz_p^{-p}r^{p-2},
\qquad
\partial_{rr}V_p(r,s)=-p(p-1)z_p^{-p}r^{p-2},
\qquad
\Lambda_p(x,y)=ps^{p-2}.
$$
Since $p<2$, we immediately get $\partial_rV_p(r,s)/r < \partial_{rr}V_p(r,s)$. Using the fact that
$r<z_ps$ and
$$z_p^2\leq p-1$$ which follows from
\cite[Corollary 3.3]{OsekowskiEJP2011}, we have
$$
-\partial_{rr}V_p(r,s)
=p(p-1)z_p^{-p}r^{p-2}
\geq p(p-1)z_p^{-2}s^{p-2}
\geq ps^{p-2}
=\Lambda_p(x,y).
$$
This proves the scalar comparisons in both regions and concludes the proof.
\end{proof}
The following proposition controls the jump contribution $\mathcal I_2$.
\begin{proposition}\label{jump prop}
Let $(h,k) \in \R^{2n}$ and $(x,y) \in \R^{2n} \setminus S$. Then 
\[
\mathcal J_\B(x,y;h,k) \leq \frac12\Lambda_p(x,y)|k|^2.
\]
\end{proposition}
Assuming Proposition~\ref{jump prop}, compensation of the jumps gives
\[
 \E\mathcal I_2\leq\frac12\E\int_0^T
       \Lambda_p(X_{s-},Y_{s-})\,d\langle Y\rangle_s.
\]
Combining the two estimates with $\Lambda_p\geq0$ and $d\langle Y\rangle\leq d\langle X^c\rangle$, we obtain
\[
 \E(\mathcal I_1+\mathcal I_2)
 \leq\frac12\E\int_0^T\Lambda_p(X_{s-},Y_{s-})
       \bigl(d\langle Y\rangle_s-d\langle X^c\rangle_s\bigr)
 \leq0.
\]
Finally, the conclusion follows from the obstacle condition
\cite[Section 4.1]{OsekowskiEJP2011}
$$
\B(x,y)\geq \mathcal{O}(x,y):=|y|^p-z_p^{-p}|x|^p.
$$ Indeed, after localization and passage to the limit, It\^o's formula gives
$$
\mathbb E\B(X_T,Y_T)\leq\mathbb E\B(X_0,Y_0)\leq0.
$$
We have then obtained 
$$
\mathbb E|Y_T|^p\leq z_p^{-p}\mathbb E|X_T|^p.
$$
Taking $p$-th roots and the supremum over $T$ yields the result. We next prove Proposition~\ref{jump prop}.

\subsection{Proof of Proposition~\ref{jump prop}}
The quadratic correction in the second variable suggests using its
square as a coordinate. Fix $(x,y)\notin S$ and set
\[
r=|x|,\qquad s=|y|,\qquad
\rho=|x+h|,\qquad \sigma=|y+k|.
\] Define
$$
F_p(r,t):=V_p(r,\sqrt t),\qquad r,t\geq0.
$$
For $r,s>0$, the chain rule gives
$$
\partial_rF_p(r,s^2)=\partial_rV_p(r,s),\qquad
\partial_tF_p(r,s^2)=\frac{\partial_sV_p(r,s)}{2s}=\Lambda_p(x,y)/2.
$$
Under this change of variables, the jump estimate becomes
\begin{align*}
\mathcal J_\B(x,y;h,k)-\frac12\Lambda_p(x,y)|k|^2
&=F_p(\rho,\sigma^2)-F_p(r,s^2)
  -\partial_rF_p(r,s^2)\frac{(x,h)}r -\partial_tF_p(r,s^2)(\sigma^2-s^2).
\end{align*}
Since
\[
\rho-r\geq\frac{(x,h)}r,
\qquad
\partial_rF_p(r,s^2)=\partial_rV_p(r,s)\leq0,
\]
we can conclude 
\[\mathcal J_\B(x,y;h,k)-\frac12\Lambda_p(x,y)|k|^2\leq F_p(\rho,\sigma^2)-F_p(r,s^2)
  -\partial_rF_p(r,s^2)(\rho-r)-\partial_tF_p(r,s^2)(\sigma^2-s^2).\]
Thus it suffices to prove that $F_p$ is concave.
\begin{lemma}\label{lem:squared-variable-concavity}
For $1<p\leq2$, $F_p$ is concave on $(0,\infty)^2$ and nonincreasing
in $r$.
\end{lemma}
\begin{proof}
For $p=2$ the claim is immediate, so let $1<p<2$. In the obstacle region,
$$
F_p(r,t)=t^{p/2}-z_p^{-p}r^p,
$$
which has a negative semidefinite Hessian since $p/2<1<p$.
In the parabolic-cylinder region, put $q=r/\sqrt t$ and
$G(q)=\phi_{p-2}(q)$. The identities in
Lemmas~\ref{parabcyl} and~\ref{radial derivatives} give
$$
\partial_{rr}F_p(r,t)
=\alpha_pp(p-1)t^{(p-2)/2}G(q),
\qquad
\partial_tF_p=-\frac12\partial_{rr}F_p.
$$
Differentiating the second identity, we get
$\partial_{rt}F_p=-\frac12\partial_{rrr}F_p$ and
$\partial_{tt}F_p=\frac14\partial_{rrrr}F_p$, hence
$$
D^2F_p(r,t)
=\alpha_pp(p-1)t^{(p-2)/2}
\begin{pmatrix}
G(q)&-\dfrac{G'(q)}{2\sqrt t}\\[3pt]
-\dfrac{G'(q)}{2\sqrt t}&\dfrac{G''(q)}{4t}
\end{pmatrix}.
$$
The integral representation
\cite[proof of Lemma 3.1]{OsekowskiEJP2011}
$$
G(q)=\frac1{\Gamma(2-p)}
\int_0^\infty v^{1-p}e^{-qv-v^2/2}\,dv
$$
shows that the matrix in this expression is positive semidefinite:  its quadratic form at $(a,b)\in\R^2$ equals
$$
\mathsf Q(a,b)=\frac1{\Gamma(2-p)}
\int_0^\infty
\left(a+\frac{bv}{2\sqrt t}\right)^2
v^{1-p}e^{-qv-v^2/2}\,dv\geq0.
$$
Since $\alpha_p<0$, we have $D^2F_p\leq0$ in both smooth regions. \\ We now pass to concavity in the whole open quadrant by an easy continuity argument. Fix
$(r_0,t_0),(r_1,t_1)$ in the open quadrant and let $(r_\theta, t_\theta)$ be a point on the line segment connecting them, for $\theta \in [0,1].$ The segment meets the singular set $S \cap (0,\infty)^2$ at most twice, that is when
$$r_\theta^2-z_p^2t_\theta=0,$$  In between, the
segment lies in one of the smooth regions, hence the $C^1$ function 
\[
G(\theta)=F_p(r_\theta,t_\theta),
\qquad 0\leq\theta\leq1,
\]
verifies
\[
G''(\theta)
=D^2F_p(r_\theta,t_\theta)[v,v]\leq0,
\qquad v=(r_1-r_0,t_1-t_0).
\]
By continuity, $G'$ is nonincreasing on the whole interval, so $F_p$ is concave in the open quadrant and, by continuity,
on $[0,\infty)^2$.
Lemma~\ref{radial derivatives} and the continuity across $S$ give
\[
\partial_rF_p(r,t)=\partial_rV_p(r,\sqrt t)\leq0,
\qquad r,t>0.
\]
\end{proof}

\begin{proof}[Proof of Proposition~\ref{jump prop}]
Fix $(x,y)\notin S$ and set
$$
r=|x|,\qquad s=|y|,\qquad
r_1=|x+h|,\qquad s_1=|y+k|.
$$
By Lemma~\ref{lem:squared-variable-concavity}, we get
$$
\B(x+h,y+k)-\B(x,y)
\leq\partial_rV_p(r,s)(r_1-r)
+\frac12\Lambda_p(x,y)(s_1^2-s^2).
$$
Since
$$
r_1-r\geq\frac{(x,h)}r,\qquad
s_1^2-s^2=2(y,k)+|k|^2,
\qquad \partial_rV_p(r,s)\leq0,
$$
we can conclude
\begin{align*}
\B(x+h,y+k)-\B(x,y)
\leq D_x\B(x,y)h+D_y\B(x,y)k
+\frac12\Lambda_p(x,y)|k|^2.
\end{align*}
\end{proof}

\subsection{Strictness}
Let $1<p<2$ and $C=z_p^{-1}>1$. Assume that
$0<\|X\|_p<\infty$ and suppose that
$\|Y\|_p=C\|X\|_p$. Following
\cite[Section 4.4]{OsekowskiEJP2011}, we first show that equality
forces the martingales to remain on the free boundary.
Set $$U_t:=\B(X_t,Y_t).$$ The growth bound
$|\B(x,y)|\leq C_p(|x|^p+|y|^p)$ and Doob's inequality, combined with the compensated drift estimates, show that $U$ is a uniformly integrable supermartingale. Writing
$\mathcal O(x,y):=|y|^p-C^p|x|^p,$ the obstacle inequality
$\mathcal O\leq\B$ and the initial domination imply
$$
0=\E\mathcal O(X_\infty,Y_\infty)
\leq\E U_\infty
\leq\E U_0
\leq0.
$$
Consequently, every inequality in this chain is an equality.
In particular, the random variable
$$
G:=\B(X_\infty,Y_\infty)-\mathcal O(X_\infty,Y_\infty)
$$
is zero almost surely.
For $1<p<2$, the obstacle inequality is strict whenever
$|y|<C|x|$; see \cite[Section 4.1]{OsekowskiEJP2011}.
Therefore, we necessarily have that
$
|Y_\infty|\geq C|X_\infty|
$ almost surely and that $\mathcal O(X_\infty,Y_\infty)$ is nonnegative. Since its expectation is zero,
$$
|Y_\infty|=C|X_\infty|.
$$
The Bellman function vanishes on the free boundary, so
$U_\infty=0$ almost surely. Using that $U_t$ is a supermartingale gives that $U_t=0$
almost surely for every fixed $t$, which implies that $U$ vanishes identically. Since
$$
\B(x,y)=0
\quad\Longleftrightarrow\quad
|y|=C|x|,
$$
we have obtained that
$$
|Y_t|^2=C^2|X_t|^2
\qquad\text{for all }t\geq0.
$$
The identity $|Y_0|=C|X_0|$ and the initial
domination $|Y_0|\leq|X_0|$ imply $X_0=Y_0=0$, since $C>1$.
For $t >0$, a standard localization argument and the uniqueness of the Doob-Meyer decomposition imply that
$$
\langle Y\rangle_{t}
=
C^2\langle X\rangle_{t}.
$$
Combining this identity
with predictable subordination and the decomposition of
$\langle X\rangle$, we find
$$
C^2d\langle X\rangle
=
d\langle Y\rangle
\leq d\langle X^c\rangle
\leq d\langle X\rangle.
$$
Since $C>1$ it follows that
$\langle X\rangle=0$, and hence also $\langle Y\rangle=0$.
Therefore, both martingales vanish identically, which is a contradiction. This proves strictness for $1<p<2$.
For $p>2$, put $q=p'$. If $\|Y\|_p=0$, strictness is immediate.
Otherwise, choose
$$
Z:=\frac{|Y_\infty|^{p-2}Y_\infty}{\|Y_\infty\|_p^{p-1}},
\qquad Z_t:=\E[Z\mid\F_t].
$$
Then $\|Z\|_q=1$ and $\E(Y_\infty,Z)=\|Y\|_p$. The reverse
construction in Proposition~\ref{duality} gives a purely discontinuous martingale
$W$ such that
$$
d\langle W\rangle\leq d\langle Z^c\rangle,\qquad
|W_0|\leq|Z_0|,\qquad
\E(Y_\infty,Z)=\E(X_\infty,W_\infty).
$$
Strictness in the lower range with exponent $q$ yields
$\|W\|_q<z_q^{-1}$. Finally, by H\"older's inequality
$$
\|Y\|_p\leq\|X\|_p\|W\|_q<z_q^{-1}\|X\|_p.
$$

\subsection{Sharpness}
Let $M$ be the parabolic Az\'ema martingale starting at zero
\cite{Emery1989Azema}, satisfying the structure equation
\[
 d[M]_t=dt-2M_{t-}\,dM_t.
\]
We use its properties
\[
 \langle M\rangle_t=t,\qquad M^c=0,\qquad |M_t|=\sqrt t.
\]
Take an independent standard Brownian motion $B$. For any nonzero bounded
stopping time $\tau$ of $B$, the martingales
$X=B^\tau$ and $Y=M^\tau$, in the joint filtration, satisfy
\[
 \langle Y\rangle_t=t\wedge\tau=\langle X^c\rangle_t,
 \qquad
 \frac{\|Y\|_p}{\|X\|_p}
 =\frac{\|\sqrt\tau\|_p}{\|B_\tau\|_p}.
\]
For $1<p\leq 2$, Davis's optimal stopping theorem
\cite{Davis1976} states that \[\sup_{\tau \in \mathcal T}\frac{\|\sqrt\tau\|_p}{\|B_\tau\|_p}=z_p^{-1}, \] where $\mathcal T$ is the set of nonzero bounded stopping times of $B$.
Thus the constant in \eqref{B.2} is sharp already for real-valued
martingales and Proposition~\ref{duality}
proves sharpness in \eqref{B.1}.


\section{Applications}
\label{sec:applications}

\subsection{The Riesz vector on the Hamming cube}
Let $\Ham_n=\{-1,1\}^n$ carry the uniform probability measure
$\mu_n$. For $x\in\Ham_n$, let $x^{(j)}$ be obtained by changing
the sign of its $j$-th coordinate. Define
\[
 D_jf(x):=\frac{f(x)-f(x^{(j)})}{2},\qquad
 \Delta=\sum_{j=1}^nD_j,
\]
since $D_j^2=D_j$. For
$S\subseteq[n]=\{1,\ldots,n\}$, the Walsh characters
$w_S(x)=\prod_{j\in S}x_j$ satisfy
\[
 D_jw_S=\mathbf1_{\{j\in S\}}w_S,\qquad
 \Delta w_S=|S|w_S.
\]
Let $P_y=e^{-y\sqrt\Delta}$ and, for $f:\Ham_n\to\R$, set
\[
 u(x,y)=P_yf(x)
 =\sum_{S\subseteq[n]}e^{-y\sqrt{|S|}}\widehat f(S)w_S(x),
 \qquad y\geq0,
\]
where $\widehat f(S)=\langle f,w_S\rangle_{L^2(\mu_n)}$.
Then $u(x,0)=f(x)$ and $u$ is harmonic, that is
\[
 (\partial_{yy}-\Delta)u=0.
\]
The Riesz transforms and Riesz vector are then
\[
 R_j=D_j\Delta^{-1/2},\qquad
 \mathbf Rf=(R_1f,\ldots,R_nf),
\]
where $\Delta^{-1/2}$ is defined to vanish on constants.

\subsubsection{The horizontal process}
To ease the reader, we explicitly construct a continuous-time Markov chain whose generator is
$-\Delta$. The resulting process $Z$ is the continuous-time random walk on $\Ham_n$ in which each coordinate flips at rate $1/2$; see
\cite[Section 2.1]{Eskenazis2023HeatFlow}. \\ 
Let $N^1,\ldots,N^n$ be independent Poisson processes
of rate $1/2$, all starting at zero, and let $Z_0$ have law
$\mu_n$, independently of these processes. Define
\[
 Z_t=(Z_t^1,\ldots,Z_t^n),\qquad
 Z_t^j=Z_0^j(-1)^{N_t^j}.
\]
When $N^j$ jumps, the $j$-th coordinate of $Z$ changes sign
and all the other coordinates remain unchanged; notice that distinct Poisson processes almost surely do not jump simultaneously. Thus $Z$ has c\`adl\`ag,
piecewise constant paths with finitely many jumps on every bounded
time interval. The independent increments of $N$ give
\[
 Z_{t+h}^j=Z_t^j(-1)^{N_{t+h}^j-N_t^j}.
\]
and the transition
kernel is
\[
 K_h(x,z):=\prob(Z_{t+h}=z\mid Z_t=x)
 =\prod_{j=1}^n\frac{1+e^{-h}x_jz_j}{2}.
\]
Therefore, for every $\varphi:\Ham_n\to\R$,
\begin{align*}
 \mathcal L_H\varphi(x)
 =\lim_{h\downarrow0}
   \frac{\E[\varphi(Z_{t+h})\mid Z_t=x]-\varphi(x)}h
 =\frac12\sum_{j=1}^n
   \bigl(\varphi(x^{(j)})-\varphi(x)\bigr)
 =-\Delta\varphi(x).
\end{align*}

One can easily prove that $\mu_n$ is invariant, and the choice $Z_0\sim\mu_n$ makes $Z$
stationary.

\subsubsection{The Poisson martingales}
Let $B$ be a standard Brownian motion, independent of Z. For $a>0$, define
\[
 Y_t^a=a+\sqrt2B_t,\qquad
 \tau=\inf\{t\geq0:Y_t^a=0\}.
\]
We suppress the superscript $a$ below. The normalization gives
$d\langle Y\rangle_t=2\,dt$, so the vertical generator is
$\partial_{yy}$. By independence, the generator of $(Z,Y)$
is $-\Delta+\partial_{yy}$ and since $u$ is harmonic
\[
 X_t=u(Z_{t\wedge\tau},Y_{t\wedge\tau})
\]
is a bounded martingale; its decomposition is computed explicitly
below.
Since $\tau<\infty$ almost surely, its terminal value is
$X_\tau=f(Z_\tau)$. Conditioning on $\tau$, which is independent
of the entire horizontal process, and using stationarity gives
$Z_\tau\sim\mu_n$. Thus, for $1\leq p<\infty$,
\[
 \|X\|_p=\|X_\tau\|_p=\|f\|_{L^p(\mu_n)}.
\]
Define the continuous $\R^n$-valued martingale $M$ associated to the Riesz vector by
\[
 M_t^j=\int_0^{t\wedge\tau}D_ju(Z_{s-},Y_s)\,dY_s,
 \qquad 1\leq j\leq n.
\]
The argument is standard, so we only sketch it. Let $g:\Ham_n\to\R$
set $v(x,y)=P_yg(x)$, and let
$G_t=v(Z_{t\wedge\tau},Y_{t\wedge\tau})$ be its Poisson
martingale. Since $M_0=0$ and $M$ is continuous
\begin{align*}
 \E\bigl[M_\tau^jg(Z_\tau)\bigr] =\E\langle M^j,G\rangle_\tau =2\int_0^\infty(a\wedge y)
       \langle D_jP_yf,\partial_yP_yg\rangle_{L^2(\mu_n)}\,dy.
\end{align*}
Here we used stationarity, independence, and the Green kernel
$a\wedge y$ of the vertical process killed at zero. The terminal pairing is justified by the bracket computation below,
which gives
\[
 \E|M_\tau|^2
 =\E\langle M\rangle_\tau
 =\E\langle X^d\rangle_\tau
 \leq\E|X_\tau|^2<\infty.
\] Recalling
\[
 -2\sqrt\lambda\int_0^\infty(a\wedge y)e^{-2y\sqrt\lambda}\,dy
 =-\frac{1-e^{-2a\sqrt\lambda}}{2\sqrt\lambda},
 \qquad \lambda>0.
\]
and expanding $f$ and $g$ in Walsh characters, with $\lambda=|S|$,
therefore gives
\[
 \E\bigl[M_\tau^jg(Z_\tau)\bigr]
 =-\frac12\langle R_j(I-P_{2a})f,g\rangle_{L^2(\mu_n)}.
\]
Since $Z_\tau\sim\mu_n$, we conclude that
\begin{equation}\label{eq:hamming-representation}
 \E[M_\tau\mid Z_\tau=x]
 =-\frac12\mathbf R(I-P_{2a})f(x).
\end{equation}
\begin{proposition}\label{prop:hamming-representation}
The martingales above satisfy
$d\langle M\rangle=d\langle X^d\rangle$. Consequently, for
$2\leq p<\infty$,
\[
 \|\mathbf Rf\|_{L^p(\Ham_n;\ell_n^2)}
 \leq\frac2{z_{p'}}\|f\|_{L^p(\Ham_n)}.
\]
\end{proposition}
\begin{proof}
We first compute the jumps and the martingale decomposition of
$X$. If $N^j$ jumps at time $s\leq\tau$, then
$Z_s=(Z_{s-})^{(j)}$, while $Y_s=Y_{s-}$ by continuity. Hence
\begin{align*}
 \Delta X_s
 &=u((Z_{s-})^{(j)},Y_s)-u(Z_{s-},Y_s)
 =-2D_ju(Z_{s-},Y_s).
\end{align*}
It\^o's formula therefore gives
\begin{align*}
 X_t-X_0
 &=\int_0^{t\wedge\tau}\partial_yu(Z_{s-},Y_s)\,dY_s
   +\int_0^{t\wedge\tau}\partial_{yy}u(Z_{s-},Y_s)\,ds
-2\sum_{j=1}^n\int_0^{t\wedge\tau}
                  D_ju(Z_{s-},Y_s)\,dN_s^j.
\end{align*}
Write the compensated Poisson process $\widetilde N_t^j=N_t^j-t/2.$ Replacing
$dN_s^j$ by $d\widetilde N_s^j+ds/2$ cancels the
drift, since $u$ is harmonic. Thus
\begin{align*}
 X_t^c=\int_0^{t\wedge\tau}\partial_yu(Z_{s-},Y_s)\,dY_s,\qquad 
 X_t^d=-2\sum_{j=1}^n\int_0^{t\wedge\tau}
                    D_ju(Z_{s-},Y_s)\,d\widetilde N_s^j.
\end{align*}
Since $\Delta N_s^j \in \{0,1\}$ and different Poisson processes do not jump simultaneously we get
\[
 [X^d]_t=\sum_{0<s\leq t}|\Delta X_s|^2
 =4\sum_{j=1}^n\int_0^{t\wedge\tau}
             |D_ju(Z_{s-},Y_s)|^2\,dN_s^j.
\]
On the other hand, using $\ang{\widetilde N^j}_s=s/2$ and $\ang{Y}_s=2s$ we get
\[
 \langle X^d\rangle_t
 =2\int_0^{t\wedge\tau}\sum_{j=1}^n |D_ju(Z_{s-},Y_s)|^2\,ds, \quad  \langle X^c\rangle_t =2\int_0^{t\wedge\tau}|\partial_yu(Z_{s-},Y_s)|^2\,ds.
\]
and
\[
 \langle M\rangle_t
 =\sum_{j=1}^n\langle M^j\rangle_t
 =2\int_0^{t\wedge\tau}\sum_{j=1}^n
             |D_ju(Z_{s-},Y_s)|^2\,ds
 =\langle X^d\rangle_t.
\]
Applying Theorem~\ref{main theorem2}, Jensen's inequality and
\eqref{eq:hamming-representation} give
\[
 \frac12\|\mathbf R(I-P_{2a})f\|_{L^p(\Ham_n;\ell_n^2)}
 \leq\|M_\tau\|_p
 \leq z_{p'}^{-1}\|X_\tau\|_p
 =z_{p'}^{-1}\|f\|_{L^p(\mu_n)}.
\]
For $a\to\infty$, we have $P_{2a}f\to\E_{\mu_n}f$ and $\textbf{R}$ annihilates constants, so the claim follows.
\end{proof}

\begin{remark}
The argument compares the continuous target energy with the
predictable compensator of the horizontal jump energy. Classical
differential subordination would instead require
\[
 \sum_{j=1}^n|D_ju(x,y)|^2\leq|\partial_yu(x,y)|^2.
\]
This fails in general: for $n=2$ and $f(x)=x_1-x_2$, we have
$u(x,y)=e^{-y}(x_1-x_2)$, so at $x=(1,1)$ the right-hand side
vanishes whereas the left-hand side equals $2e^{-2y}$.
\end{remark}

\subsection{The Riesz vector on $\Z^n$}

Equip $\Z^n$ with counting measure and define
\[
 \partial_j^+f(x)=f(x+e_j)-f(x),\qquad
 \partial_j^-f(x)=f(x)-f(x-e_j),
\]
\[
 \Delta f(x)=\sum_{j=1}^n
 \bigl(2f(x)-f(x+e_j)-f(x-e_j)\bigr),\qquad
 R_j^\pm=\partial_j^\pm\Delta^{-1/2}.
\]
We consider the $2n$-dimensional vector
\[
 \mathbf Rf=(R_1^+f,R_1^-f,\ldots,R_n^+f,R_n^-f).
\]
Let $Z$ be the random walk jumping to each neighboring point at
rate one, and let $Y$
and $\tau$ as above. Initially take $f$ finitely
supported, set $u=P_yf$, and define
\[
 X_t=u(Z_{t\wedge\tau},Y_{t\wedge\tau}),\qquad
 M_t^{j,\pm}=\int_0^{t\wedge\tau}
       \partial_j^\pm u(Z_{s-},Y_s)\,dY_s.
\]
The jump and continuous bracket computations give
\begin{align*}
 d\langle X^d\rangle_t
 =\mathbf1_{\{t\leq\tau\}}
   \sum_{j=1}^n
   \bigl(|\partial_j^+u|^2+|\partial_j^-u|^2\bigr)
   (Z_{t-},Y_t)\,dt,\qquad 
 d\langle M\rangle_t=2\,d\langle X^d\rangle_t.
\end{align*}
Thus Theorem~\ref{main theorem2} applies to $M/\sqrt2$ and $X$. Writing $\E_{x,a}$ for expectation with $Z_0=x$ and $Y_0=a$,
invariance of counting measure gives
\[
 \sum_{x\in\Z^n}\E_{x,a}|f(Z_\tau)|^p
 =\|f\|_{\ell^p(\Z^n)}^p.
\]
The usual Poisson
pairing, summed over starting points, yields
\[
 \frac12\|\mathbf R(I-P_{2a})f\|_{\ell^p(\Z^n;\ell_{2n}^2)}
 \leq
 \left(\sum_{x\in\Z^n}\E_{x,a}|M_\tau|^p\right)^{1/p}
 \leq\frac{\sqrt2}{z_{p'}}\|f\|_{\ell^p(\Z^n)}.
\]
Letting $a\to\infty$ as before, we obtain the result.

\section{Further applications and remarks}

The main idea of this paper is to study the continuous and jump drift terms together, allowing their defects to compensate at the predictable level. The same mechanism leads to extensions of other martingale inequalities under weaker hypotheses, which are the subject of ongoing work by the author and some collaborators. We now comment on possible further applications.

Consider a Markov jump-diffusion process $Z$ with generator \(\mathcal L=\mathcal L_c+\mathcal L_j\). Define
\[
\Gamma_i(v)=\frac12\bigl(\mathcal L_i(v^2)-2v\mathcal L_iv\bigr),
\qquad i=c,j.
\]
Let \(u\) be \(\mathcal L\)-harmonic. Under suitable domain and integrability assumptions, the martingale \(X_t=u(Z_t)\), with stopping if necessary, satisfies
\[
d\langle X^c\rangle_t=2\Gamma_c(u)(Z_{t-}),dt,
\qquad
d\langle X^d\rangle_t=2\Gamma_j(u)(Z_{t-}),dt.
\]
Suppose that a continuous martingale $Y$, starting at zero, has bracket
\[
d\langle Y\rangle_t=2|Gu(Z_{t-})|^2dt,
\]
for a suitable gradient or transform \(G\). For \(2\leq p<\infty\), the pointwise comparison
\[
|Gu|^2\leq K^2\bigl(\Gamma_c(u)+\Gamma_j(u)\bigr),
\]
therefore Theorem \ref{main theorem1} applies. If the stronger comparison
\[
|Gu|^2\leq K^2\Gamma_j(u)
\]
holds, then Theorem \ref{main theorem2} applies.
Classical subordination would require control by \(K^2\Gamma_c(u)\) alone. These criteria identify which part of the source energy is available to control the target and which sharp martingale estimate can be applied. In the study of the Riesz vector in the Dunkl setting \cite{dunklmine}, for example, Theorem~\ref{main theorem1} applies to the natural martingale representation, while classical differential subordination fails and Theorem~\ref{main theorem2} does not apply in general to that representation. The example in Section~\ref{sec: discussion of hypothesis} also suggests further applications to comparisons between stochastic integrals driven by Brownian motion and compensated Poisson processes. Predictable subordination allows such comparisons even when their different pathwise behavior prevents classical differential subordination. 
\subsection{AI disclosure} The author used ChatGPT to assist with exposition and technical calculations, mostly to simplify earlier versions of the proof of \cref{jump prop}. Also, while the author was seeking a discontinuous martingale that could reproduce, in norm, the behavior arising in Davis’s stopping-time problem \cite{Davis1976}, an AI-assisted literature search identified parabolic Azéma martingales, previously unknown to the author. This helped prove the sharpness. The formulation of the results and the proof strategies are solely due to the author and draw inspiration from previous work in \cite{dunklmine}. All AI-assisted material was carefully checked by the author, who takes full responsibility for the contents of the paper.
\bibliographystyle{amsalpha}
\bibliography{biblio}
\end{document}